\documentclass[11pt]{amsart}
\usepackage{pstricks,pst-plot}
\usepackage{amssymb,amsmath}

\parskip=\smallskipamount

\newtheorem{theorem}{Theorem}

\newtheorem*{proposition}{Proposition}

\theoremstyle{definition}

\newtheorem*{remark}{Remark}

\newcommand{\B}{\mathbb{B}}
\newcommand{\C}{\mathbb{C}}

\newcommand{\Z}{\mathbb{Z}}

\newcommand{\R}{\mathbb{R}}

\newcommand{\PP}{\mathrm P}

\newcommand{\Grr}{{\mathrm{Gr}}_\R}
\newcommand{\Grc}{{\mathrm{Gr}}_\C}
\newcommand{\Gl}{{\mathrm{Gl}}}

\newcommand{\cA}{\mathcal{A}}

\newcommand{\cC}{\mathcal{C}}

\newcommand{\cM}{\mathcal{M}}

\newcommand\hra{\hookrightarrow}

\newcommand\Real{\mathrm{Re}}

\newcommand\grad{\mathrm{grad}}
\newcommand\dist{\mathrm{dist}}

\def\bs{\backslash}
\def\bar{\overline}

\numberwithin{equation}{section}

\begin{document}
\title[Modeling complex points up to isotopy]
{Modeling complex points up to isotopy}
\author{Marko Slapar}
\address{University of Ljubljana, Faculty of Education, Kardeljeva Plo\v s\v cad 16, 1000 Ljubljana, Slovenia}
\email{marko.slapar@pef.uni-lj.si}
\thanks{Supported by the research program P1-0291 at the Slovenian Research Agency.}

%
%
\date{\today} 

\begin{abstract}
In this paper we examine the structure of complex points of real $4$-manifolds embedded into complex $3$-manifolds up to isotopy. We show that there are only two types of complex points up to isotopy and as a consequence, show that any such embedding can be deformed by isotopy to a manifold having $2$-complete neighborhood basis. 
\end{abstract}
\maketitle

\section{Introduction} 
Let $i\!:Y\hra X$ be a  real compact $2n-2$ dimensional manifold $Y$, smoothly embedded into an $n$ dimensional complex manifold $(X,J)$. If the embedding is sufficiently generic, all but finitely many points $p\in Y$ are CR regular and have the dimension of the maximal complex tangent subspace $T^\C_pY=T_pY\cap JT_pY$  equal to $n-2$. For dimension reason, at the finite set of CR singular points, the tangent space $T_pY=T^\C_pY$ and we call such points {\it complex points} of $Y$. If $Y$ is an oriented manifold, the orientation of $T_pY$ can be compared with the induced orientation of $T_pY$ as a complex subspace of $T_pX$. If these two orientations agree, the complex point is called {\it positive}, if not, {\it negative}. 

Using local coordinates on $X$, we can assume that locally $i\!:Y\hra \C^n$. Complex points are then exactly the points that are being mapped into $\Grc(n-1,\C^n)\subset \Grr(2n-2,\C^n)$ by the Gauss map $Di\!:p\mapsto T_pY\in \Grr(2n-2,\C^n)$. Using a version of Thom transversality theorem \cite{Thom}, these intersections  are transverse for generic embeddings. Depending on the sign of these intersections, we call complex points {\it elliptic} (positive sign) or {\it hyperbolic} (negative sign). Note that the terms elliptic and hyperbolic here are used differently than by Dolbeault, Tomassini and Zaitsev in \cite{DTZ1,DTZ2}. Using this terminology, we algebraically count complex points as $I=e-h$, where $e$ is the number of elliptic complex points and $h$ the number of hyperbolic complex points. If $Y$ is oriented, we can also introduce $I_\pm=e_\pm-h_\pm$, taking into account the sign of complex points. The indices $I,I_\pm$ are called Lai indices. It is evident from the construction that these indices are invariant under isotopies and there are topological index formulas, calculating these invariants, see \cite{Lai}. 





In local holomorphic coordinates, defined in a neighborhood $U$ of an isolated complex point $p\in Y$, the manifold $Y$ can we written as
\begin{equation}\label{form} w=\bar z^TAz+\Real (z^TBz)+o(|z|^2),\end{equation}
where $(z,w)$ are coordinates in $\C^n=\C^{n-1}\times\C$, $z=(z_1,z_2,\ldots,z_{n-1})$, and $A,B$ are $(n-1)\times(n-1)$ complex matrices. The above form can be easily derived from a general Taylor expansion, since we can do a holomorphic coordinate change $w\mapsto w-p(z)$ for any holomorphic polynomial in $z$. The matrix $B$ can be assumed to be symmetric. 
If we calculate the intersection index using local coordinates, see \cite{C1}, one can see that the point $p$ is elliptic, if the determinant of
\begin{equation}\label{det}\left[
\begin{array}{cc}
A & \bar B \\
B & \bar A
\end{array}\right]
\end{equation}
is positive and hyperbolic, if it is negative. We will call a pair  $(A,B)$ (or a complex point itself) {\it nondegenerate}, if the above determinant is not $0$. A complex point is called {\it quadratic}, if in some local coordinates, the term $o(|z|^2)$ vanishes; {\it flat}, if in some local coordinates near the complex point, the manifold $Y$ can be put in $\C^{n-1}\times\R$, and  {\it quadratically flat} if that is true up to the quadratic part of (\ref{form}).

The case $n=2$ is well understood. In some local holomorphic coordinates, complex points are given by the equation $w=z\bar z+\gamma(z^2+{\bar z}^2)+o(|z|^2)$ with $0\le\gamma\le\infty$, where $\gamma=\infty$ is understood as $w=z^2+{\bar z}^2+o(|z|^2)$. The parameter $\gamma$ is called the Bishop invariant  and the points are hyperbolic for $\gamma>\frac{1}{2}$ and elliptic for $0\le\gamma<\frac{1}{2}$, see \cite{Bis}. Complex point in dimension $2$ are always quadratically flat, which is not the case in higher dimensions. In dimension $2$ in the real analytic case, elliptic points are flat and hyperbolic points are formally flat for all but countably many $\gamma>1/2$, see \cite{MoW}. Flatness and more generally normal forms beyond quadratic terms are not well understood in higher dimensions. They have been studied for spherical models in \cite{HY} and \cite{Bur}. In dimension $n=2$, a pair of hyperbolic and elliptic complex point can be canceled by a $\cC^0$ small isotopy (as long as both points are of the same sign, if the surface is oriented) \cite{EH, F}, and the surface with only hyperbolic flat hyperbolic complex points has tubular Stein neighborhood basis \cite{S}. For most complex surfaces (for example $X$ Stein or $X$ of general type), any smoothly embedded real surface, except spheres with trivial homology class, can be deformed by a small smooth isotopy to only have  hyperbolic flat complex points and thus have Stein neighborhood basis. These results are derived using the Seiberg-Witten adjunction inequality \cite{FS,OS,LiM,N}. In the case of $n>2$ one cannot get Stein neighborhoods for purely topological reasons. The best one can expect are $2$ complete neighborhoods ($2$ strictly positive directions of the Levi form of a defining function) and we show below that after a smooth isotopy, this can indeed be done.  We will treat the cancellation theorem in higher dimensions in a subsequent paper. There is one more direction we would like to mention. In dimension $n=2$ there is a one dimensional family of Bishop discs shrinking toward an elliptic complex point, see \cite{Bis}. Such a family foliates a Levi flat hypersurface.  An analogous problem of finding a Levi flat hypersurface bounded by a codimension $2$ real submanifold in higher dimensions was studied in \cite{DTZ1,DTZ2}. 

In this paper, we prove the following two theorems

\begin{theorem}\label{thm1} Every smooth embedding of a $4$ dimensional compact real manifold $Y$ into a $3$ dimensional complex manifold $X$ can be deformed by a smooth isotopy to a manifold only having complex point of the following two types:
\begin{itemize}
\item[i)] $w=|z_1|^2+|z_2|^2$,
\item[ii)] $w=|z_1|^2+\overline{z_2}^2$.
\end{itemize}
The isotopy is $\cC^0$ close to the original embedding.
\end{theorem}

\begin{remark} In the above theorem elliptic complex points are deformed to be of type i) and hyperbolic to be of type ii).  We can take other suitable model situations for complex points. For elliptic points, we can for example take $w=|z_1|^2-|z_2|^2$ or $w=\overline{z_1}^2+\overline{z_2}^2$.\end{remark}

 \begin{theorem}\label{thm2} Let $Y\hookrightarrow X$ be a smooth embedding of a compact $4$-manifold in a complex $3$ manifold. After a $\cC^0$ small smooth isotopy, $Y$ has a $2$-complete tubular neighborhood basis.  
\end{theorem}

In the above theorem, we call a neighborhood tubular, if it diffeomorphic to the normal bundle of $Y$. We expect that analogous statements also hold in higher dimensions.

\section{Normal forms of quadratic complex points of real $4$-manifolds in $\C^3$.}

Local normal forms of $4$-manifolds in $\C^3$ have been studied and are listed by Coffman in \cite{C1}. We review part of the material here, both for completeness and because we would like to know which  normal forms arise generically. 

Let $p$ be a complex point of a real $4$-manifold $Y$ in a complex $3$-manifold given  in some local coordinates by  (\ref{form}), $(z,w)\in \C^2\times\C$, $z=(z_1,z_2)$, and $A,B$ some $2\times 2$ matrices with $B^T=B$.  Any local holomorphic change of coordinates preserving both the complex point and the tangent space has the same effect on the quadratic part of the equation (\ref{form}) as the linear change 
$$\begin{pmatrix}
z \\ 
w 
\end{pmatrix}
=
\begin{pmatrix}
\PP & *\\
0 & c_{n,n}
\end{pmatrix}
\begin{pmatrix}
\tilde z\\
\tilde w
\end{pmatrix}.
$$

\noindent This linear change transforms the equation into 

$$w=\frac{1}{c_{nn}}\bar z^TP^*APz+\frac{1}{2}\left(\frac{1}{c_{nn}}z^TP^TBPz+\frac{1}{c_{nn}}{\bar z}^T{\bar P}^T\bar B\bar P\bar z\right)+o(|z|^2),$$

\noindent where we have dropped the $\sim$ in the new coordinates. After a quadratic holomorphic change in $w$, the equation becomes

$$w=\frac{1}{c_{nn}}\bar z^TP^*APz+\frac{1}{\overline{c_{nn}}}\Real (z^TP^TBPz)+o(|z|^2).$$

We introduce the group action of $S^1\times \Gl(2,\C)/\Z_2$ (the $\Z_2$ quotient means $P\sim-P$), on the space of pairs of matrices $M=\{(A,B);\ A,B\in M_2(\C),\ B=B^T\}$ by
$$(\zeta,P)(A,B)=(\zeta P^*AP,\bar \zeta P^TAP)$$ and let $\cM$ be the quotient of $M$ by this action. We have seen that the space $\cM$ is the moduli space of quadratic complex points up to biholomorphic change of coordinates. We call two pairs $(A_0,B_0),\ (A_1, B_1)\in M$ {\it h-congruent} if they are in the same orbit of this group action.  Let us denote by $M^+$ the part of $M$ where the determinant of (\ref{det}) is positive and by $M^-$ the part of $M$ where it is negative. Since the sign of the determinant is preserved by the group action, we can also denote by $\cM^+$ and $\cM^-$ the quotients of $M^\pm$ by the group action. These are the moduli spaces of elliptic and hyperbolic quadratic complex points.

To obtain a normal form for quadratic complex points under biholomorphic coordinate change, we need to look at canonical forms for h-congruence. The first step is to find canonical forms for *congruence. This is done in \cite{HS} for $n\times n$ complex matrices and we do here a quick review in the case of $n=2$. Let $A$ be a $2\times 2$ complex matrix and we assume that $A$ is non singular (this is a generic assumption for complex points). Then the matrix $\cA=A^{-*}A$ (called the *cosquare of $A$) has the property $A\cA A^{-1}=\cA^{-*}$. Since $\cA$ is similar to $\cA^{-*}$,  we have only the following options for the Jordan form of $\cA$:
$$
 \bullet\begin{pmatrix} e^{i\alpha} &0\\0&e^{i\beta}\end{pmatrix}\quad\bullet \begin{pmatrix} e^{i\alpha} &1\\0&e^{i\alpha}\end{pmatrix}\quad\bullet
 \begin{pmatrix} \mu&0\\0& 1/{\bar \mu}\end{pmatrix},\ 0<|\mu|<1.
$$
It the first case, the matrix $A$ turns out to be *congruent to one of the matrices $\pm\bigl(\begin{smallmatrix} e^{i\alpha/2} &0\\0&\pm e^{i\beta/2}\end{smallmatrix}\bigr)$, in the second case it is *congruent to $\pm {e^{i\alpha/2}}\bigl(\begin{smallmatrix} 0 &1\\1&i\end{smallmatrix}\bigr)$ and in the third case, $A$ is *congruent to $\bigl(\begin{smallmatrix} 0 &1\\\mu &0\end{smallmatrix}\bigr)$. For h-congruence, we also have an option of multiplying the matrix $A$ by a nonzero complex number, so we can further simplify the form of $A$ to one of the following forms  
\begin{equation}\label{123}
{\mathrm {i)}}\ \begin{pmatrix} 1 &0\\0&e^{i\theta}\end{pmatrix}, 0\le\theta\le \pi\quad 
{\mathrm {ii)}}\ \begin{pmatrix} 0 &1\\1&i\end{pmatrix}\quad 
{\mathrm {iii)}}\ \begin{pmatrix} 0 &1\\ \mu&0\end{pmatrix},\ 0<\mu<1. 
\end{equation}
One gets $0\le\theta\le\pi$ in $\mathrm{i)}$ since diagonal elements can be interchanged by the congruence with the matrix $\bigl(\begin{smallmatrix}0&1\\1&0\end{smallmatrix}\bigr)$. The parameter $\mu$ in $\mathrm{ii)}$ can be made real by first doing a *congruence using the matrix $\bigl(\begin{smallmatrix}1&0\\0&x\end{smallmatrix}\bigr)$ for a suitable $x$, followed by the division by $x$. The details, using a slightly different construction, can also be found in (\cite{C2}, Theorem 4.3). 

If the matrix $\cA$ has only one eigenvalue $e^{i\alpha}$, we get cases $\theta=0$ or $\theta=\pi$ from $\mathrm{i)}$, as long as the eigenspace is two dimensional. If it is one dimensional, we get the case $\mathrm{ii)}$. If the matrix $\cA$ has two distinct eigenvalues, we either get $0<\theta<\pi$ in $\mathrm{i)}$ or the case $\mathrm{iii)}$. Let as see that in a  generic situation, $\cA$ has two distinct eigenvalues. For some $A=\bigl(\begin{smallmatrix}a&b\\c&d\end{smallmatrix}\bigr)$ we have
$$\cA=A^{-*}A=\frac{1}{\overline{ad-bc}}\begin{pmatrix}a\bar d-|c|^2&b\bar d-\bar c d\\\bar a c-a\bar b&\bar a d-|b|^2\end{pmatrix}$$
and  it's characteristic polynomial is
$$p_{\cA}(\lambda)=\lambda^2-\frac{a\bar d+\bar a d-|c|^2-|b|^2}{\overline{ad-bc}}\lambda+\frac{ad-bc}{\overline{ad-bc}}.$$
If the characteristic polynomial has two distinct zeros $\lambda_1,\lambda_2$ with $|\lambda_1|=|\lambda_2|=1$, then $|\lambda_1+\lambda_2|=|a\bar d+\bar a d-|c|^2-|b|^2|/|{ad-bc}|<2$; if the two distinct zeros are $\lambda_1=\mu$, $\lambda_2=\frac{1}{\bar\mu}$,
 then $|\lambda_1+\lambda_2|=|a\bar d+\bar a d-|c|^2-|b|^2|/|{ad-bc}|>2$. So outside a codimension $1$ real subvariety $|a\bar d+\bar a d-|c|^2-|b|^2|=2|ad-bc|$, the matrix $A$ is *congruent to either $\mathrm{i)}$ with $\theta\ne0,\pi$ or $\mathrm{iii)}$. One could also see that the case $\mathrm{ii)}$ account for most of the matrices with $|a\bar d+\bar a d-|c|^2-|b|^2|/|{ad-bc}|=2$, while the cases $\ \theta=0,\pi$ from $\mathrm{i)}$ are of further real codimension $2$.
 
One can check that the only elements $(\zeta,P)$ of the group $S^1\times \Gl(2,\C)/\Z_2$, so that $\zeta P^*\bigl(\begin{smallmatrix}1&0\\0&e^{i\theta}\end{smallmatrix}\bigr)P=\bigl(\begin{smallmatrix}1&0\\0&e^{i\theta}\end{smallmatrix}\bigr)$, $0<\theta<\pi$, are $\left(1,\bigl(\begin{smallmatrix}e^{i\alpha}&0\\0&e^{i\beta}\end{smallmatrix}\bigr)\right)$. We can use this group elements to simplify the second matrix $B$ in the pair $(A,B)$ 
so that the diagonal elements of $B$ become nonnegative. If we assume the generic situation where both diagonal elements of $B$ are nonzero, these are essentially the only simplifications. Similarly, the only elements $(\zeta,P)$ that preserve $\bigl(\begin{smallmatrix}0&1\\ \mu&0\end{smallmatrix}\bigr)$ for $0<\mu<1$  are $\left(1,\bigl(\begin{smallmatrix}a&0\\0&1/\bar a\end{smallmatrix}\bigr)\right)$ for some nonzero complex $a$. 
If the diagonal elements of $B$ are both nonzero, these group elements can be used to make the diagonal element of $B$ conjugate to each other, and in this case, these are again the only simplifications. So in a generic situation, we can assume that for any complex point the pair $(A,B)$ in (\ref{form}) is h-congruent to one of the following 
\begin{itemize}
\item $A=\begin{pmatrix}1&0\\0&e^{i\theta}\end{pmatrix},\ 0<\theta<\pi,\ B=\begin{pmatrix}a&b\\b&d\end{pmatrix},\ a,d>0$
\item $A=\begin{pmatrix}0&1\\\mu&0\end{pmatrix},\ 0<\mu<1,\ B=\begin{pmatrix}a&b\\b&\bar a\end{pmatrix},\ a\ne 0$.  
\end{itemize}

If $\theta=0,\pi$ in $\mathrm{i)}$ or in the case $\mathrm{ii)}$, the subgroup of $S^1\times \Gl(2,\C)/\Z_2$ preserving $A$ under *congruence is larger. We state the possible forms of the matrix $B$ for some of those cases in the next section.

\section{Proof of main theorems}
\begin{proposition} $M^\pm$ and $\cM^\pm$ are connected.\label{connected}\end{proposition}  
\begin{proof}
Let $(A,B)\in M$. We have seen in the previous section that, perhaps after a small homotopy, we can assume that the pair $(A,B)$ is h-congruent to one of the two types of pairs

\begin{itemize}
\item[i)] $A=\begin{pmatrix} 1&0\\0&e^{i\phi}\end{pmatrix},\ 0<\phi<\pi$, $B=\begin{pmatrix}a&b\\b&d\end{pmatrix},\ a,b\ge 0,\ b\in \C,$
\item[ii)] $A=\begin{pmatrix} 0&1\\ \tau&0\end{pmatrix},\ 0<\tau<1$, $B=\begin{pmatrix}a&b\\b&\bar a\end{pmatrix},\ a,b\in \C.$ 
\end{itemize}

Notice first that the group $S^1\times\Gl(2,\C)/\Z_2$ is connected, so we can always connect a pair $(A,B)$ to its h-congruent pair by a homotopy.  

Let us assume first that $(A,B)$ is of type i). We have

\begin{equation}\label{det1}D=\det\begin{pmatrix} 1&0&a&\bar b\\
                     0&e^{i\phi}&\bar b&d\\
										 a&b&1&0\\
										 b&d&0&e^{-i\phi}
				\end{pmatrix}=|b|^4-ad(b^2+{\bar b}^2)-2|b|^2\cos\phi+(1-a^2)(1-d^2).\end{equation}

\noindent If $D<0$, we can deform the pair $(A,B)$ by a homotopy $A_t=\bigl(\begin{smallmatrix} 1&0\\ 0&e^{(1-t)i\phi}\end{smallmatrix}\bigr)$, $B_t=B$ to a pair			
with $A=\bigl(\begin{smallmatrix}1&0\\0&1\end{smallmatrix}\bigr)$. Using the classification list (\cite{C2}, Theorem 7.2), such a pair can alway be deformed by a homotopy to an h-congruent pair 
$$A=\begin{pmatrix}1&0\\0&1\end{pmatrix},\ B=\begin{pmatrix}a&0\\0&d\end{pmatrix},\ d>1>a\ge 0.$$
We then first deform the matrix $B$ by a linear homotopy in the parameter $a$, bringing $a$ to $0$. The pair 
$A=\bigl(\begin{smallmatrix}1&0\\0&1\end{smallmatrix}\bigr)$, $B=\bigl(\begin{smallmatrix}0&0\\0&d\end{smallmatrix}\bigr)$ is h-congruent to  
$A=\bigl(\begin{smallmatrix}1&0\\0&1/d\end{smallmatrix}\bigr)$, $B=\bigl(\begin{smallmatrix}0&0\\0&1\end{smallmatrix}\bigr)$ by the element $\left(1,\bigl(\begin{smallmatrix}1&0\\0&1/\sqrt d\end{smallmatrix}\bigr)\right)$. 
The last step is to do the homotopy $A_t=\bigl(\begin{smallmatrix}1&0\\0&(1-t)/d\end{smallmatrix}\bigr)$, $B_t=B$ to get the pair
\begin{equation}\label{hyp1}A=\begin{pmatrix}1&0\\0&0\end{pmatrix},\ B=\begin{pmatrix}0&0\\0&1\end{pmatrix}.\end{equation} 
It is easy to check that all homotopies stay in $M^-$.
If $D>0$, doing a homotopy $A_t=\bigl(\begin{smallmatrix} 1&0\\ 0&e^{(1-t)\phi+t\pi}\end{smallmatrix}\bigr)$, $B_t=B$ keeps the determinant (\ref{det1}) positive and brings the pair $(A,B)$
to a pair with $A=\bigl(\begin{smallmatrix}1&0\\0&-1\end{smallmatrix}\bigr)$. Using the classification list (\cite{C2}, Theorem 7.2), such a pair is h-congruent to a pair with the same $A$, and $B$  from the list
\begin{itemize}
\item[a)] $B=\begin{pmatrix}a&0\\0&d\end{pmatrix}$, with $0\le a\le d<1$ or $1<a\le d$,
\item[b)] $B=\begin{pmatrix}1+d&1-d\\1-d&1+d\end{pmatrix}$, $\mathrm{Im}\ d>0$,
\item[c)] $B=\begin{pmatrix}0&b\\b&0\end{pmatrix}$, $b> 0$,
\item[d)] $B=\begin{pmatrix}1+b&-1\\-1&1-b\end{pmatrix}$, $b> 0$, $b\ne 1$,
\item[e)] $B=\begin{pmatrix}1&1\\1&1\end{pmatrix}$. 
\end{itemize}

\noindent We can either argue off cases c),d) and e) for not being generic, or we can just do simple homotopies. We do a small homotopy $A_t=A,\ B_t=B+t\epsilon \bigl(\begin{smallmatrix}i&-i\\-i&i\end{smallmatrix}\bigr)$ to bring the case e) to b). The case c) is deformed by the homotopy $A_t=A,\ B_t=(1-t)B$ that stays in $M^+$ to the pair 
\begin{equation}\label{eli1} A=\begin{pmatrix}1&0\\0&-1\end{pmatrix},\ B=\begin{pmatrix}0&0\\0&0\end{pmatrix}.\end{equation}
In the case d), if $b<1$, we first do a linear homotopy in the parameter $b$, bringing $b$ to $0$,  to get 
$B=\bigl(\begin{smallmatrix}1&-1\\-1&1\end{smallmatrix}\bigr)$. 
The homotopy stays in $M^+$. The pair $(A,B)$ is h-congruent by 
$\left(-1,\bigl(\begin{smallmatrix}1&1+i\\1-i&1\end{smallmatrix}\bigr)\right)$ to the case e). If $b>1$, the pair $(A,B)$ is h-congruent to the pair $A=\bigl(\begin{smallmatrix}\frac{1}{1+b}&0\\0&\frac{-1}{1+b}\end{smallmatrix}\bigr)$,  $B=\bigl(\begin{smallmatrix}1&\frac{-1}{1+b}\\\frac{-1}{1+b}&-1+\frac{2}{1+b}\end{smallmatrix}\bigr)$. Next we use the homotopy $A_t=(1-t)A$,  $B_t=\bigl(\begin{smallmatrix}1&\frac{t-1}{1+b}\\ \frac{t-1}{1+b}&-1+\frac{2-2t}{1+b}\end{smallmatrix}\bigr)$ to get to 
$A=\bigl(\begin{smallmatrix}0&0\\0&0\end{smallmatrix}\bigr),\ B=\bigl(\begin{smallmatrix}1&0\\0&-1\end{smallmatrix}\bigr)$. The homotopy stays in $M^+$, and this last pair is h-congruent to 
\begin{equation}\label{eli2}
A=\begin{pmatrix}0&0\\0&0\end{pmatrix},\ B=\begin{pmatrix}1&0\\0&1\end{pmatrix} 
\end{equation}
by the group element $(1,\bigl(\begin{smallmatrix}1&0\\0&i\end{smallmatrix}\bigr).$
The case b) is reduced to the case a) after a linear homotopy only in the parameter $d$, bringing $d$ to $1$. In the case a), if both $a,b< 1$, the homotopy $A_t=A,\ B_t=(1-t)B$ stays in $M^+$ and brings the pair $(A,B)$ to (\ref{eli1}). If $a,d>1$, then $(A,B)$ is h-congruent to $A=\bigl(\begin{smallmatrix}1/a&0\\0&-1/d\end{smallmatrix}\bigr),B=\bigl(\begin{smallmatrix}1&0\\0&1\end{smallmatrix}\bigr)$ by the group element $\left(1,\bigl(\begin{smallmatrix}1/\sqrt{a}&0\\0&1/\sqrt{d}\end{smallmatrix}\bigr)\right)$. A further homotopy $A_t=(1-t)A,\ B_t=B$  brings the pair inside $M^+$ to (\ref{eli2}). 

We now assume $(A,B)$ is of type ii.) We have
\begin{equation}\label{det2}\begin{split} D&=\det\begin{pmatrix} 0&1&\bar a&\bar b\\\
                     \tau&0&\bar b&a\\
										 a&b&0&1\\
										 b&\bar a&\tau&1
				\end{pmatrix}=\\
				&=|b|^4-|b|^2\tau^2-|b|^2-2|a|^2|b|^2\cos2\beta+|a|^4-2|a|^2\tau\cos2\alpha+\tau^2,\end{split}\end{equation}
where $a=|a|e^{i\alpha}$ and $b=|b|e^{i\beta}$. 

\noindent If $D>0$, a linear homotopy in the pair $(\alpha,\beta)$, bringing $\alpha,\beta$ to either $\pm\pi/2$, deforms the matrix $B$ to $\bigl(\begin{smallmatrix}ia&ib\\ib&-ia\end{smallmatrix}\bigr),\ a,b\in \R$, and the path stays in $M^+$. The determinant (\ref{det2}) becomes
$$D=|b|^4+|b|^2(2|a|^2-\tau^2-1)+(|a|^2+\tau)^2.$$	
If we treat the above formula as a quadratic function in $|b|^2$, it will either have two positive zeros $x_1\le x_2$ (in the case $\tau^2+1\ge 2|a|^2\ge (1-\tau)^2/2$), or will be positive for all $b$. In the second case, we bring $b$ to $0$ by  a linear homotopy in the parameter $b$. In the first case, if $|b|^2< x_1$,  we also just bring the parameter $b$ to $0$. If $|b|^2>x_2$, the pair $(A,B)$ is h-congruent to the pair $A=\bigl(\begin{smallmatrix}0&1/|b|\\ \tau/|b|&0\end{smallmatrix}\bigr),\ B=\bigl(\begin{smallmatrix}ia/|b|&i\\i&-ia/|b|\end{smallmatrix}\bigr).$ We follow by the homotopy $A_t=\bigl(\begin{smallmatrix}0&(1-t)/|b|\\ (1-t)\tau/|b|&0\end{smallmatrix}\bigr)$, $B_t=\bigl(\begin{smallmatrix}(1-t)ia/|b|&i\\i&(t-1)ia/|b|\end{smallmatrix}\bigr)$ to get the pair $A=\bigl(\begin{smallmatrix}0&0\\0&0\end{smallmatrix}\bigr),$ $B=\bigl(\begin{smallmatrix}0&i\\i&0\end{smallmatrix}\bigr)$. The homotopy stays in $M^+$. This pair is h-equivalent to (\ref{eli2}) by the element $\left(-i, \frac{1}{\sqrt 2}\bigl(\begin{smallmatrix}1&i\\1&-i\end{smallmatrix}\bigr)\right)$.
		
\noindent If $D<0$, the linear homotopy in $\alpha$ and $\beta$ that either brings them to $0$ or $\pi$,  makes $B=\bigl(\begin{smallmatrix}a&b\\b&a\end{smallmatrix}\bigr),\ a,b\in \R$, stays in $M^-$, and makes the determinant (\ref{det2}) equal to
$$D=|b|^4-|b|^2(2|a|^2+\tau^2+1)+(|a|^2-\tau)^2.$$ 
The above equation, as a quadratic function in $|b|^2$, always has two positive zeros $x_1\le x_2$. Since we assume that $D<0$, the value $|b|^2$ is between $x_1$ and $x_2$. We do a linear homotopy in $b$ bringing $|b|^2$ to $\frac{x_1+x_2}{2}=|a|^2+\frac{\tau^2+1}{2}$.
The determinant $D$ becomes $(|a|^2-\tau)^2-\frac{1}{4}(2|a|^2+\tau^2+1)^2$, which is always negative. We follow by linear homotopies in parameters $a$ and $\tau$, bringing both of them to $0$. The induced homotopy $(A_t,B_t)$ stays in $M^-$ and brings $(A,B)$  
to the pair 
$A=\bigl(\begin{smallmatrix}0&1\\0&0\end{smallmatrix}\bigr)$, $B=\bigl(\begin{smallmatrix}0&1/\sqrt{2}\\ 1/\sqrt{2}&0\end{smallmatrix}\bigr).$  We follow by the homotopy $A_t=A,\ B_t=\bigl(\begin{smallmatrix}0&(1-t)/\sqrt{2}\\ t+(1-t)/\sqrt{2}&0\end{smallmatrix}\bigr)$, which stays in $M^-$, 
to get $A=\bigl(\begin{smallmatrix}0&1\\0&0\end{smallmatrix}\bigr)$, $B=\bigl(\begin{smallmatrix}0&0 \\ 1&0\end{smallmatrix}\bigr).$ At last we do the homotopy $A_t=\bigl(\begin{smallmatrix}t+ix&(1-t)\\0&0\end{smallmatrix}\bigr),\ B_t=\bigl(\begin{smallmatrix}0&0\\ 1-t&t\end{smallmatrix}\bigr),$ where $x(t)$ is some real function, compactly supported on $(0,1)$. The 
determinant $\det\bigl(\begin{smallmatrix}A_t&\bar B_t\\B_t&\bar A_t\end{smallmatrix}\bigr)$ equals $-((1-2t)^2+t^2x^2)$, so the homotopy stays completely in $M^-$ as long as $x(\frac{1}{2})\ne 0$. So we get that the original pair $(A,B)$ is homotopic  to (\ref{hyp1}) inside $M^-$.

The only thing we still need to show is that the two pairs (\ref{eli1}) and (\ref{eli2}) can also be joined by a homotopy. We can first see that (\ref{eli1}) can be deformed to 
\begin{equation}\label{eli3}
A=\begin{pmatrix}1&0\\0&1\end{pmatrix},\ B=\begin{pmatrix}0&0\\0&0\end{pmatrix} 
\end{equation}
 by the homotopy $A_t=\bigl(\begin{smallmatrix}1&0\\0&e^{i(1-t)\pi}\end{smallmatrix}\bigr),\ B_t=\bigl(\begin{smallmatrix}0&0\\0&0\end{smallmatrix}\bigr)$, which stays in $M^+$. The homotopy $A_t=\bigl(\begin{smallmatrix}1-t&e^{i\pi/4}x\\-e^{i\pi/4}x&1-t\end{smallmatrix}\bigr),$ $B_t=\bigl(\begin{smallmatrix}t&0\\0&t\end{smallmatrix}\bigr),$ where $x(t)$ is some small compactly supported smooth real function on $(0,1)$,  connects the pairs (\ref{eli3}) and (\ref{eli2}). If $x(\frac{1}{2})\ne 0$, the homotopy stays in $M^+$ since the determinant $\det\bigl(\begin{smallmatrix}A_t&\bar B_t\\B_t&\bar A_t\end{smallmatrix}\bigr)$ equals $(1-2 t)^2+|x|^2(|x|^2+2t^2)$. 

We have shown that all nondegenerate pairs $(A,B)$ can be connected to one of the two model types (\ref{hyp1}) or (\ref{eli2}), depending on the sign of the pair. So $M^\pm$ are both connected. The spaces $\cM^\pm$ are then also connected, since they are quotients of $M^\pm$.
\end{proof}

\begin{proof}[Proof of theorem \ref{thm1}] Perhaps after a small isotopic perturbation of $Y\hookrightarrow X$, we can assume that all complex points on $Y$ are either elliptic or hyperbolic. Let now $p$ be a complex point on $Y$. In local coordinates near $p$, the manifold $Y$ is given by an equation
$$w=\bar{z}^TA_0z+\Real (z^TB_0z)+o(|z|^2).$$ We can also assume that $o(|z|^2)=0$ since we can do a small isotopy $w=\bar{z}^TA_0z+\Real (z^TB_0z)+(1-t\phi(|z|))o(|z|^2)$, where $\phi\!:[0,\epsilon) \mapsto [0,1]$ has compact support and is $1$ near $0$. Let now $(A(t),B(t))$ be a smooth homotopy, $(A(t),B(t))=(A_1,B_1)$ near $t=0$ and $(A(t),B(t))=(A_0,B_0)$ near $t=1$. Using the proposition above,  we can assume that the homotopy stays in the same $M^\pm$ part as $(A_0,B_0)$. That means that 
\begin{equation}\label{nondeg} \det\begin{pmatrix}A(t)&\overline B(t)\\B(t)&\overline A(t)\end{pmatrix}\ne 0\,\quad \forall t\in\left[0,1\right]. \end{equation} Let $\varepsilon$ be small and let $\tilde Y$ coincide with $Y$ away from $|z|<\varepsilon$ and is given by
$$w=\bar{z}^TA(\sqrt[n]{|z|/\varepsilon})z+\Real (z^TB(\sqrt[n]{|z|/\varepsilon})z)$$
for $|z|\le\varepsilon$.  $Y$ and $\tilde Y$ are of course isotopic. The complex point of the manifold $\tilde Y$ at $0$ is modeled by the pair $(A_1,B_1)$. We need to show that we have not created any new complex points. A point $(z,w)$ on a manifold given by $w=f(z)$ is a complex point if and only if 
$$\frac{\partial f}{\partial \overline{z_1}}(z)=\frac{\partial f}{\partial \overline{z_2}}(z)=0.$$ 
In our case, using $B=B^T$, there are no new complex points created if and only if for $z\ne 0,\ |z|\le\varepsilon$
\begin{align*}
 &A(\sqrt[n]{|z|/\varepsilon})z+\bar B(\sqrt[n]{|z|/\varepsilon})\bar z+\\&+\sqrt[n]{\frac{|z|}{\varepsilon}}\frac{1}{2n|z|^2}\left({\overline {z}}^TA'(\sqrt[n]{|z|/\varepsilon})z +\Real ({z}^TB'(\sqrt[n]{|z|/\varepsilon})z)\right)z\ne 0
\end{align*}
We write $z=|z|z'$, where $|z'|=1$. The above equation is equivalent to
\begin{align*}
 &A(\sqrt[n]{|z|/\varepsilon})z'+\bar B(\sqrt[n]{|z|/\varepsilon})\overline {z'}+\\&+\sqrt[n]{|z|/\varepsilon}\frac{1}{2n}\left({\overline {z'}}^TA'(\sqrt[n]{|z|/\varepsilon})z'  +\Real ({z'}^TB'(\sqrt[n]{|z|/\varepsilon})z')\right)z'\ne 0,
\end{align*} 
 or using $s=\sqrt[n]{\frac{|z|}{\varepsilon}}$
 \begin{equation}\label{complex}
 A(s)z'+\bar B(s)\overline {z'}+\frac{s}{2n}\left({\bar z}^TA'(s)z+\Real \left({z'}^TB'(s)z'\right)\right)z'\ne 0,
\end{equation}
for all $|z_1|=1,s\in[0,1].$ If there exists a $z'$, $|z'|=1$, so that $A(s)z'+\bar B(s)\overline {z'}=0$ for some $s$, then the pair $(z',\overline{z'})$ is in the kernel of $\begin{pmatrix}A(t)&\overline B(t)\\B(t)&\overline A(t)\end{pmatrix}$. Since (\ref{nondeg}) holds, this cannot happen. Because of compactness, 
$||A(s)z'+\bar B(s)\overline {z'}||>\delta$ for some $\delta>0$, and all $s\in\left[0,1\right]$ and $z',\ |z'|=1$. Let $m=\max |{\bar z}^TA'(s)z+\Real \left({z'}^TB'(s)z'\right)|$. If we choose $n>m/(2\delta)$, the inequality (\ref{complex}) holds.  
 \end{proof}

\noindent The following proposition, together with Theorem \ref{thm1}, proves Theorem \ref{thm2}.

\begin{proposition} Let $Y$ be a smoothly embedded compact $4$-manifold in a complex $3$-manifold $X$, having only finitely many complex points that are all modeled by $w=\bar{z_1}^2+\bar{z_2}^2$ or $w=|z_1|^2+\bar{z_2}^2$. Then there exists a neighborhood $U$ of $Y$ and a smooth positive function $\phi$ with the properties  
\begin{itemize}
\item $\phi\ge 0$ and $\{\phi=0\}=Y$,
\item $\phi$ is strictly $2$-convex except at complex points, 
\item there exist positive constants $c,C$ so that $c (\dist(q,Y))^2\le\phi(q)\le C (\dist(q,Y))^2$ and $c \dist(q,Y)\le |\grad \phi(q)|\le C \dist(q,Y)$.  
\end{itemize} 
where $\dist(q,Y)$ is the distance of $q$ to $Y$ in some smooth metric on $X$.
\end{proposition}

\begin{proof} Let $Y\hookrightarrow X$. Let $V$ be some small open set in $X$ containing all complex points. The result (\cite{Ch}, Proposition 6.5) provides such a function in a neighborhood of $Y\bs V$. So all we need is to construct suitable functions in small neighborhoods of complex points and then patch them together with the the function for $Y\bs V$ using the partition of unity. 
\noindent Let first $p$ be a complex point of the form $w=\bar{z_1}^2+\bar{z_2}^2$, and let 
$$f(z,w)=(1+|z|^2)|w-\bar{z_1}^2-\bar{z_2}^2|^2.$$ We show that $f$ is strictly plurisubharmonic in a small neighborhood of $(0,0)$, except at the origin. The Levi form of $f$ equals
$$Lf=\left(\begin{smallmatrix}4|z_1|^2(1+|z|^2)+|u|^2-2(z_1^2u+\bar{z_1}^2\bar u) & 4z_1\bar{z_2}(1+|z|^2)-2(z_1z_2u+\bar{z_1}\bar{z_2}\bar u) & \bar{z_1}u\\
 4\bar{z_1}z_2(1+|z|^2)-2(z_1z_2u+\bar{z_1}\bar{z_2}\bar u)&4|z_2|^2(1+|z|^2)+|u|^2-2(z_2^2u+\bar{z_2}^2\bar u)&\bar{z_2}u\\
 z_1\bar{u}&z_2\bar{u}&1+|z|^2\end{smallmatrix}\right),$$ 
where $u=w-\bar{z_1}^2-\bar{z_2}^2$. Using Sylvester's criterion (and a long calculation), one can check that the matrix $Lf$ is positive definite in a neighborhood of the origin by calculating determinants of the minors. The determinant only vanishes at the origin. 

\noindent In the case of $p$ being of type $w=|z_1|^2+\bar{z_2}^2$, we define 
 $$f(z,w)=(1+|z_2|^2)|w-|z_1|^2-\bar{z_2}^2|^2$$ and we get the Levi form
 $$Lf=\left(\begin{smallmatrix}2|z_1|^2-(u+\bar u)(1+|z_2|^2)& 2\bar{z_1}\bar{z_2}(1+|z_2|^2)-\bar{z_1}z_2(u+\bar u) & -\bar{z_1}(1+|z_2|^2)\\
 2z_1z_2(1+|z_2|^2)-z_1\bar{z_2}(u+\bar u)& |u|^2-2(z_2^2u+\bar{z_2}^2\bar u)+4|z_2|^2(1+|z_2|^2)&\bar{z_2}u\\
 -z_1(1+|z_2|^2)& z_2\bar{u}& 1+|z_2|^2\end{smallmatrix}\right),$$
 where $u=w-|z_1|^2-\bar{z_2}^2$. 
One quickly sees that the trace of $Lf$ is strictly positive in a neighborhood of the origin, so $Lf$ has at least one strictly positive eigenvalue. The determinant of the minor 
$$\begin{pmatrix}
  |u|^2-2(z_2^2u+\bar{z_2}^2\bar u)+4|z_2|^2(1+|z_2|^2)&\bar{z_2}u\\
 z_2\bar{u}& 1+|z_2|^2\end{pmatrix}$$
equals 
\begin{align*}|u|^2-&2(z_2^2u+\bar{z_2}^2\bar u)(1+|z_2|^2)+4|z_2|^2(1+|z_2|^2)^2=(|u|-2|z_2|^2(1+|z_2|^2))^2 \\                &+2(1+|z_2|^2)(2|u||z_2|^2-z_2^2u-\bar{z_2}^2\bar u)+4|z_2|^2(1+|z_2|^2)^2(1-|z_2|^2).
 \end{align*}
This is clearly strictly positive, except for the points $\{z_2=u=0\}.$ It is an easy check that for $\{z_2=u=0\}$ the Levi form $Lf$  has two strictly positive eigenvalues, except at the origin. For other points in a neighborhood of the origin, the Levi form is strictly positive if the determinant of $Lf$ is strictly positive, or has exactly one non-positive eigenvalue if the determinant is negative. 
\end{proof}

\bibliographystyle{amsplain}

\end{document}